\documentclass[10pt,a4paper]{article}
\usepackage{typearea}
\usepackage{amsmath}
\usepackage{amsfonts}
\usepackage{amssymb}
\usepackage{amsthm}
\usepackage{amscd}
\typearea{11}
\author{Andreas Klein}
\title{Symplectic Spinors, Holonomy and Maslov Index}
\newtheorem{theorem}{Theorem}[section]
\newtheorem{Def}[theorem]{Definition}
\newtheorem{prop}[theorem]{Proposition}
\newtheorem{lemma}[theorem]{Lemma}
\newtheorem{folg}[theorem]{Corollary}

\begin{document}
\maketitle
\begin{abstract}
In this note it is shown that the Maslov Index for pairs of Lagrangian Paths as introduced by Cappell, Lee and Miller appears by parallel transporting elements of (a certain complex line-subbundle of) the symplectic spinorbundle over Euclidean space, when pulled back to an (embedded) Lagrangian submanifold $L$, along closed or non-closed paths therein. In especially, the CLM-Index mod $4$ determines the holonomy group of this line bundle w.r.t. the Levi-Civita-connection on $L$, hence its vanishing mod $4$ is equivalent to the existence of a trivializing parallel section. Moreover, it is shown that the CLM-Index determines parallel transport in that line-bundle along arbitrary paths when compared to the parallel transport w.r.t. to the canonical flat connection of Euclidean space, if the Lagrangian tangent planes at the endpoints either coincide or are orthogonal. This is derived from a result on parallel transport of certain elements of the dual spinorbundle along closed or endpoint-transversal paths.
\end{abstract}

\section{Introduction} The idea that (some kind of) Maslov Index is related to
the double covering of the symplectic group, called the metaplectic group and
to the notion of symplectic spinors has been implicit in the literature
for quite a long time, mainly in the context of geometric quantization (see
Guillemin/Sternberg \cite{guillemin}, Kostant \cite{kost} and Crumeyrolle \cite{cru2}).
More recent work of Gosson (\cite{gosson}), who gives an analytical definition
of a maslov index (mapping to $\mathbb{Z}_4$) on the metaplectic group using
its well-known Shale-Weil representation enlightened this area considerably.
Using this and, to get in touch with some common definition of Maslov index,
its link to the Maslov Index for pairs of Lagrangian Paths as discussed by
Cappell, Lee and Miller in their well known paper \cite{CLM}, the announced
result is little more than `piecing the edges together'.
To give a brief outline of the argument, let $(V,\omega)$ be a fixed
(finite dimensional) symplectic vectorspace and let $Lag(V)$ be the space of
Lagrangian subspaces in $V$. To a continuous and piecewise
smooth path $f(t)=(L_1(t),L_2(t)), t \in [a,b]$ in $Lag(V)\times Lag(V)$ let
the 
\[
{\rm Maslov \ index} \quad \mu_{V,CLM}(f)
\] 
be the integer invariant associated to $f$ following \cite{CLM}, from now on
referred to as CLM-index. The CLM-index is characterized by a set of axioms
which include homotopy invariance relative fixed endpoints. This is the reason
why there is an associated index $\mathcal{M}(x,y)$ for a pair $(x,y)$ in the
universal covering space $\pi:\widetilde{Lag}(V) \rightarrow Lag(V)$ of
$Lag(V)$. In fact, choose a path $\tilde{\gamma}:[0,1] \rightarrow
\widetilde{Lag}(V)$ such that
\[
\tilde{\gamma}(0)=x,\quad \tilde{\gamma}(1)=y.
\]
If $\gamma=\pi\tilde{\gamma}$ then for any Lagrangian $L_0$ in $V$ the integer 
\begin{equation}\label{CLMdef}
\mathcal{M}(L_0;x,y)=\mu_{V,CLM}([L_0],\gamma)
\end{equation}
where $[L_0]$ is the constant path, is well-defined. One chooses
$L_0=\gamma(1)$, so from now on we refer to
$\mathcal{M}(x,y)=\mathcal{M}([\gamma(1)];x,y)$ as the Maslov-index on pairs
of the universal covering space $\widetilde{Lag}(V)$. Now, as we shall see
below, the usual action of the symplectic group $Sp(V)$ of $V$ on $Lag(V)$ is
covered by an action of the universal covering group $\widetilde{Sp}(V)$ of
$Sp(V)$ on $\widetilde{Lag}(V)$ 
\[
\widetilde{Sp}(V) \times \widetilde{Lag}(V) \rightarrow \widetilde{Lag}(V).
\]
For a fixed Lagrangian $L \in Lag(V)$ one now chooses an element $\tilde{L} \in
\widetilde{Lag}(V)$ with $\pi(\tilde{L})=L$ and observes that
\[
m_{L}(\tilde{S})=\mathcal{M}(\tilde{S}\tilde{L},\tilde{L})+n
\]
where $\tilde{S} \in \widetilde{Sp}(V)$ does not depend on the choice of
$\tilde{L}$ ($n=\frac{dim(V)}{2}$ is convention). Hence the last expression
defines a $\mathbb{Z}$-valued mapping on $\widetilde{Sp}(V)$ associated to $L$.
Finally, since $\widetilde{Sp}(V)$ covers the metaplectic group $\rho: Mp(V)
\rightarrow Sp(V)$, say $\pi_2:\widetilde{Sp}(V)\rightarrow Mp(V)$, one defines
for $S \in Mp(V)$ \[
m_{L}(S)= m_{L}(\tilde{S}) \mod 4
\]
where $\pi_2(\tilde S)=S$ and shows that one gets a well defined mapping
$m_{L}:Mp(V)\rightarrow \mathbb{Z}_{4}$. Now, specializing to $(V,\omega)$ as
$(\mathbb{R}^{2n},\omega_{0})$, with $\omega_{0}$ the symplectic standard
structure on $R^{2n}$, Gosson (\cite{gosson}) shows, that for
$L_{0}= \{0\} \times \mathbb{R}^{n}$ one recovers the index $m_{L_{0}}$ on
$Mp(2n,\mathbb{R})$ using analytic properties of the Shale-Weil-representation
of the metaplectic group, this will be the key to our proof.\\
Given a Lagrangian embedding in $(\mathbb{R}^{2n},\omega_{0})$, that is a
manifold $L$ with ${\rm dim}(L)=n$ and an embedding $i:L \rightarrow
\mathbb{R}^{2n}$ with $i^{*}\omega_0=0$, we will look at the pullback
$i^{*}\mathcal{Q}_{0}$ of a certain complex one-dimensional subbundle
$\mathcal{Q}_{0}$ of $\mathcal{Q}$, the symplectic spinorbundle over
$(\mathbb{R}^{2n},\omega_{0})$, to $L$. Further we will consider the parallel transport
$\mathcal{P}^{\nabla^{g}}$ in $i^*\mathcal{Q}_{0}$ induced by the
Levi-Civita connection $\nabla^{g}$ of the Riemannian metric $g$ on $L$ which
makes $i$ isometric relative to the standard metric on $\mathbb{R}^{2n}$. It is known that 
\begin{equation}\label{root}
i^*\mathcal{Q}_{0}\otimes i^*\mathcal{Q}_{0} \simeq i^*\Lambda^{-1}, 
\end{equation}
where $i^*\Lambda$ denotes the canonical bundle on $\mathbb{R}^{2n}$, pulled back to $L$. We then have the following result.
\begin{theorem}\label{theorem1}
For smooth closed paths $\gamma: [0,1] \rightarrow L$ based at $x \in L$ we have 
\begin{equation}\label{theorem}
\mathcal{P}^{\nabla^{g}}_{\gamma}\varphi=e^{i\frac{\pi}{2}\mu_{CLM}([i_{*}T_{x}L],[i_{*}\gamma])}\varphi, 
\end{equation}
if $\varphi \in (\mathcal{Q}_{0})_x$, where $[i_{*}T_{x}L]$ is the
corresponding constant path and $[i_*\gamma]$ is the path $t \mapsto i_{*}(T_{\gamma(t)}L)$ in $Lag(\mathbb{R}^{2n})$. Consequently, for the holonomy group ${\rm Hol}^{\nabla^g}(i^*Q_0)$ we have ${\rm Hol}^{\nabla^g}(i^*Q_0)\subset \mathbb{Z}_4$.
\end{theorem}
Denote by ${\rm Par}^{\nabla^g}(i^*Q_0)\subset \Gamma(i^*\mathcal{Q}_0)$ the set of sections which are parallel w.r.t. $\nabla^g$, Theorem \ref{theorem1} implies:
\begin{folg}\label{cor1}
With the above notations we have ${\rm dim}_{\mathbb{C}}({\rm Par}^{\nabla^g})(i^*Q_0)=1$ if and only if $\mu_{CLM}([i_{*}T_{x}L],[i_{*}\gamma])=0 \mod 4$ for all $\gamma \in \pi_1(L)$.
\end{folg}
Note that from (\ref{root}) it follows that the holonomy of $i^*\mathcal{Q}_0$ is determined by a Maslov-Index (namely the value of the mean-curvature form of $L$ on $\gamma \in H_1(L,\mathbb{Z})$, see Oh \cite{oh}). However, our proof does not use (\ref{root}) and instead derives the Theorem using the Maslov index $\hat \mu$ on $Mp(2n)$. Furthermore, the approach shows that $\mu_{CLM}$ determines parallel transport in $i^*\mathcal{Q}$ w.r.t. $\nabla^g$ along non-closed paths in $L$ in an appropriate 'semi-classical limit'. To explain that, let $\nabla^0$ denote the connection on $i^*\mathcal{Q}$ induced by the canonical flat connection on $\mathbb{R}^{2n}$, extended to the dual spinor bundle $i^*\mathcal{Q}'$. Let $P_L$ be the $O(n)$-reduction of $i^*P$ 
which is induced by $L$, $P$ being the metaplectic structure of $(\mathbb{R}^{2n},\omega_0)$ (cf. Lemma \ref{onred}). Assume $\delta_p(x)\in \mathcal{Q}'_x$ assigns for a given $p\in (P_L)_x$ to any $\phi \in i^*\mathcal{Q}_x, \phi=[p,u]$ the value $\delta_p(\phi)=\delta(0)(u)$ (see (\ref{delta})) and is extended to a $\nabla^0$-parallel section $\delta_p \in \Gamma(i^*\mathcal{Q}')$. Denote by $\delta_p(y) \in (i^*Q')_y$ its restriction to $y \in L$. Analogously, let $\mathbf{1}_p \in \Gamma(i^*\mathcal{Q}')$ be the $\nabla^0$-parallel dual spinor field defined by $\mathbf{1}_p=[p, 1] \in i^*\mathcal{Q}'_x$. Then we have
\begin{theorem}\label{theorem2}
Let $\gamma:[0,1]\rightarrow L$ denote a smooth path connecting $x,y \in L$ and assume that $L(x)\perp L(y)$ w.r.t. the Euclidean metric in $\mathbb{R}^{2n}$, where $L(x)=i_*(T_xL)$ and $L(y)=i_*(T_yL)$, respectively. Then $\mathcal{P}^{\nabla^{g}}_{\gamma}\delta_p(x)\in (i^*\mathcal{Q}')_y$ and we have
\begin{equation}\label{transversal}
\mathcal{P}^{\nabla^{g}}_{\gamma}\delta_p(x)=c(y)e^{-i\frac{\pi}{2}\mu_{CLM}([L(y)],[i_{*}\gamma])}\mathbf{1}_p(y),
\end{equation}
for $0 < c(y) \in \mathbb{R}$. On the other hand, suppose that ${\rm dim}\ L(x)\cap L(y)=n$ in $\mathbb{R}^{2n}$, then 
\begin{equation}\label{othercase}
\mathcal{P}^{\nabla^{g}}_{\gamma}\delta_p(x)=e^{-i\frac{\pi}{2}\mu_{CLM}([L(y)],[i_{*}\gamma])}\delta_p(y),
\end{equation}
in $\mathcal{Q}'_y$.
\end{theorem}
Let now $\mathcal{Q}^J_l,\ l\in \mathbb{N}_0$ be the splitting of $\mathcal{Q}$ induced by the canonical complex structure $J$ of $\mathbb{R}^{2n}$, i.e. $\mathcal{Q}_0=\mathcal{Q}^J_0$ (see Section \ref{metaplect}, Prop. \ref{decomp}). Then Theorem \ref{theorem2} immediately implies
\begin{folg}\label{cor2}
Let $\gamma:[0,1]\rightarrow L$ be a smooth path with endpoints $x,y \in L$ s.t. ${\rm dim}\ L(x)\cap L(y)=n$ or $L(x)\perp L(y)$ and let $\varphi\in \Gamma(i^*\mathcal{Q}_0)$ be $\nabla^0$-parallel, then
\begin{equation}\label{transversal2}
\mathcal{P}^{\nabla^{g}}_{\gamma}\varphi(x)=e^{i\frac{\pi}{2}\mu_{CLM}([L(y)],[i_{*}\gamma])}\varphi(y).
\end{equation}
On the other hand, if ${\rm dim}\ L(x)\cap L(y)=n$ in $\mathbb{R}^{2n}$ and $\psi_l\in \Gamma(i^*\mathcal{Q}_l)$ is $\nabla^0$-parallel then
\begin{equation}\label{othercase2}
\delta_p(y)(\mathcal{P}^{\nabla^{g}}_{\gamma}\psi_l(x))=e^{i\frac{\pi}{2}\mu_{CLM}([L(y)],[i_{*}\gamma])}\delta_p(y)(\psi_l),
\end{equation}
where $\delta_p \in \Gamma(\mathcal{Q}')$ is as defined above Theorem \ref{theorem2}.
\end{folg}
Note that we give a formula extending Theorem \ref{theorem2} to endpoint transversal paths (that is $L(x)\cap L(y)=0$) in Theorem \ref{theorem3} (compare also (\ref{transversegen})). Further note that (\ref{transversal2}) extends Theorem \ref{theorem1} to the case of endpoint-orthogonal paths in $L$ in the sense that the arguments used to prove Theorem \ref{theorem1} already show that (\ref{transversal2}) holds for $L(x)=L(y)$. On the other hand, (\ref{othercase2}) means that $\mu_{CLM}$ determines the 'holonomy at zero' along closed paths in any of the subbundles $\mathcal{Q}_l$, that is, the holonomy multiplies the 'zero value' of any element of $\mathcal{Q}_l$ w.r.t to a given metaplectic frame by some element of $\mathbb{Z}_4\subset U(1)$ which is determined by the Maslov index.\\
The paper is organized as follows: in Section \ref{section1}, we will review in some more detail the above mentioned facts concerning the diverse integer invariants on Lagrangians paths and certain (cyclic) coverings of the symplectic group. Section \ref{metaplect} contains a short discussion of the metaplectic representation with special emphasis on the properties of the so called `quadratic Fourier transforms' and gives some 
necessary background on symplectic spinors. In Section \ref{proofs} finally
we will arrive at the actual proof of Theorems \ref{theorem1}, \ref{theorem2} and Theorem \ref{theorem3}.\\
We thank the anonymous referee for numerous valuable remarks and critique.

\section{Maslov indices for Lagrangian Paths and the Metaplectic Group}\label{section1}
\label{maslov}
In this section, $(V,\omega)$ will be $(\mathbb{R}^{2n},\omega_0)$ and we will
write $Lag(n)$, $\widetilde {Lag}(n)$ for the Lagrangian Grassmannian and
its universal covering, $Sp(2n)$, $Mp(2n)$ and $\widetilde{Sp}(2n)$ for the
symplectic group resp. its connected twofold and universal covering groups. To give some
intuition to the definitions we will review some fundamental results about the
Lagrangian Grassmannian, its universal covering and associated group actions,
see for instance Souriau (\cite{souriau}).\\
Since on one hand $U(n)=Sp(2n) \cap 0(2n)$ acts transitively on $Lag(n)$ 
\begin{equation}\label{action}
U(n)\times Lag(n) \rightarrow Lag(n), \quad (s,L) \mapsto sL
\end{equation}
with $O(n)$ the isotropy subgroup of $L_0=0\times \mathbb{R^n}$ we have $Lag(n)
\simeq U(n)/O(n)$. On the other hand, if $R_1,R_1 \in U(n)$ with
$R_1L_0=R_2L_0$ and the lower case letters $r_1,r_2$ denote the inverse images
of $R_1,R_2$ under the isomorphism
\[
i: U(n,\mathbb{C}) \subset M(n,\mathbb{C}) \rightarrow U(n)\subset
M(2n,\mathbb{R}) \quad (A+iB) \mapsto \left(\begin{smallmatrix}A
&-B\\B&A\end{smallmatrix} \right)
\]
where $A,B \in M(n,\mathbb{R}), \ A^TA+B^TB=I$ and $A^TB$ symmetric, then
\[
R_1L_0=R_2L_0 \ \Leftrightarrow \ r_1(r_1)^T=r_2(r_2)^T
\]
where $r^T$ is the transposed of $r$. Hence we get a homeomorphism
\[
F: Lag(n) \rightarrow W(n,\mathbb{C})=U(n,\mathbb{C}) \cap
sym(n,\mathbb{C}),\quad  L=RL_0 \mapsto rr^T
\]
satisfying $F(RL)=rF(L)r^T$, concluding that we identified $Lag(n)$ with
a subset of $U(n,\mathbb{C})$. Now the action (\ref{action}) is covered by a
unique transitive group action
\begin{equation}\label{actuni}
\tilde{U}(n,\mathbb{C}) \times \widetilde{Lag}(n) \rightarrow
\widetilde{Lag}(n) \end{equation}
where $\tilde{U}(n,\mathbb{C})$ is the universal covering group of
$U(n,\mathbb{C})$ which can easily seen to be
realized by defining 
\[\tilde{U}(n,\mathbb{C}) =\{(r,\phi):r  \in U(n,\mathbb{C}),\ 
det(r)= e^{i\phi}\}\] 
with the group composition
$(r,\phi)(r',\phi')=(rr',\phi+\phi')$ and projection $\pi:(R,\phi)\mapsto R$
and using the topology induced by $\pi$. Define
$\widetilde{W}(n,\mathbb{C})=\{(w,\phi)\in \tilde{U}(n,\mathbb{C}):w \in
W(n,\mathbb{C})\}$ with projection to $W(n,\mathbb{C})$ being the restriction
of $\pi$. Observe that $\widetilde{W}(n,\mathbb{C})$ is connected and
simply connected since the group $\tilde{U}(n,\mathbb{C})$ acts
transitively on $\widetilde{W}(n,\mathbb{C})$ with isotropy subgroup SO(n) of
$(I,0)$ by defining 
\begin{equation}\label{actconcrete}
(R,\phi)(w,\theta)=(rwr^T,\theta+2\phi),\quad (r,\phi) \in
\widetilde{U}(n,\mathbb{C}),\ (w,\theta) \in \widetilde{W}(n,\mathbb{C}). 
\end{equation} 
So $\widetilde{W}(n,\mathbb{C})\simeq \widetilde{Lag}(n)$ and
the above action realizes (\ref{actuni}) covering (\ref{action}). The
decktransformations of $\tilde{U}(n,\mathbb{C})$ are obviously of the form
${I}\times 2\pi\mathbb{Z}$, so
$\pi_1(U(n,\mathbb{C}))=\pi_1(Sp(2n))= I\times 2\pi\mathbb{Z}$.
Identifying the group of decktransformations of $\widetilde{Lag}(n)$ with the
subgroup $I \times \pi\mathbb{Z} \subset \widetilde{U}(n,\mathbb{C})$ by the action (\ref{actconcrete}), we arrive at $\pi_1(Lag)=I\times\pi\mathbb{Z}$. If we denote $\beta=(I,\pi)$ and $\alpha=(I,2\pi)$ the respective generators of $\pi_1(Lag(n))$ and
$\pi_1(Sp(2n))$ we get \begin{equation}\label{doubling}
(\alpha\widetilde{U})(\tilde{L})=\beta^2(\tilde U\tilde L)=\widetilde U(\beta^2\tilde
L) \end{equation}
for $\tilde U \in \widetilde U(n,\mathbb{C})$, $\tilde L \in \widetilde{Lag}(n)$. Understanding $\alpha$ resp. $\beta$ as generators of the group of decktransformations of $\widetilde{Sp}(2n)$ and $\widetilde{Lag}(n)$ (using that $U(n,\mathbb{C})\subset Sp(2n)$ is a maximal compact subgroup) we define for $q\in \mathbb{N}_+$
\begin{equation}
Sp_q(2n)=\widetilde{Sp}(2n)/\{\alpha^{qk}:k\in \mathbb{Z}\}\quad {Lag}_q(n)=\widetilde{Lag}(n)/\{\beta^{qk}:k\in \mathbb{Z}\}.
\end{equation}
These constitute the (unique up to isomorphism) $q$-fold cyclic connected coverings $\rho_q:{Sp}_q(2n)\rightarrow {Sp}(2n)$ resp. $\varphi_q:{Lag}_q(n)\rightarrow {Lag}(n)$. This means that $(\rho_q)_*(\pi_1({Sp}_q(2n)))=q\mathbb{Z}$ resp. $(\varphi_q)_*(\pi_1({Lag}_q(n)))=q\mathbb{Z}$ and one has the commuting diagram 
\begin{equation}\label{commutation}
\begin{CD}
\widetilde{Sp}(2n)  @>>{\pi^{Sp}_q}> {Sp}_q(2n) \\
@VV{\pi^{Sp}}V              @VV{\rho_q}V \\
{Sp}(2n) @>>id>   {Sp}(2n)
\end{CD}
\end{equation}
where $\pi^{Sp}_q$ is defined so that the diagram commutes. Note there is an analogous diagram in the case of $Lag(n)$ involving the mapping $\pi_k:\widetilde{Lag}(n)\rightarrow {Lag}_q(n)$ satisfying $\pi=\rho_q\circ\pi_q:\widetilde{Lag}(n))\rightarrow Lag(n)$. As a consequence of (\ref{doubling}), we infer that the action (\ref{actuni}) projects for each $q>0$ to an action
\begin{equation}
Sp_q(2n) \times Lag_{2q}(n) \rightarrow Lag_{2q}(n).
\end{equation}
Now, in \cite{gosson2} resp. \cite{gosson} one defines an index $\mu: \widetilde{Lag}(n)\times \widetilde{Lag}(n)\rightarrow \mathbb{Z}$ which is uniquely defined by the two conditions, where we write in the following $L=\pi(\tilde L)$ for $\tilde L \in \widetilde{Lag}(n)$:
\begin{enumerate}
\item $\mu$ is locally constant on the set $\{(\tilde L_1,\tilde L_2): L_1\cap L_2=0\}$
\item for $(\tilde L_1, \tilde L_2, \tilde L_3)\in \widetilde {Lag}^3(n)$ we have
\[
\mu(\tilde L_1,\tilde L_2)-\mu(\tilde L_1,\tilde L_3)+\mu(\tilde L_2, \tilde L_3)=\tau(L_1,L_2.L_3).
\]
\end{enumerate}
Here, $\tau$ is the signature of the quadratic form on $L_1\oplus L_2 \oplus L_3$ defined by
\[
(z_1,z_2,z_3) \mapsto \omega(z_1,z_2)\oplus \omega(z_2,z_3)\oplus\omega(z_1,z_3).
\]
As is shown in (\cite{gosson2}, Proposition 3.16 resp. Corollary 3.22), if $(\tilde S,\tilde L_1,\tilde L_2)\in Sp_q(2n) \times Lag_{2q}(n)^2$, then
\begin{equation}\label{invariance}
\mu(\tilde S\tilde L_1, \tilde S\tilde L_2)=\mu(\tilde L_1, \tilde L_2),
\end{equation}
furthermore if $\beta=(I,\pi)$ generates $\pi_1(Lag(n))$ as above, then 
\begin{equation}\label{generator}
\mu(\beta^r \tilde L_1, \beta^{r'} \tilde L_2)=\mu(\tilde L_1, \tilde L_2)+2(r-r'),
\end{equation}
if $r,r' \in \mathbb{Z}$. For $\tilde L_1,\tilde L_2$, let $\mathcal{M}(\tilde L_1,\tilde L_2)\in \mathbb{Z}$ be as defined below (\ref{CLMdef}) and define
\begin{equation}\label{maslovequality}
\hat \mu(\tilde L_1,\tilde L_2)=2\mathcal{M}(\tilde L_1,\tilde L_2) +(n-{\rm dim}(L_1\cap L_2)).
\end{equation}
Then using the defining conditions for $\mu$, it is proven in (\cite{CLM}, Prop. 9.1) that
\begin{lemma}\label{cappelmaslov}
For all $\tilde L_1,\tilde L_2 \in \widetilde{Lag}(n)$ the index $\mu(\tilde L_1,\tilde L_2)$ coincides with $\hat \mu(\tilde L_1,\tilde L_2)$.
\end{lemma}\label{spindex}
The two properties (\ref{invariance}) and (\ref{generator}) of $\mu$ imply the following Proposition resp. Definition of a Maslov index on $\widetilde{Sp}(2n)$ resp. $Sp_q(2n)$ relative to a fixed Lagrangian $L \in Lag(n)$, which was the aim of this section:
\begin{lemma}\label{modmaslov}
Let $L \in {Lag}(n)$, then the mapping $\mu: \widetilde{Sp}(2n)\rightarrow \mathbb{Z}$ given by
\[
\mu_L(\tilde S)=\mu(\tilde S\tilde L,\tilde L)
\]
is well-defined, that is, independent of the choice of $\tilde L$ lifting $L$. Furthermore, for any $q \in \mathbb{N}_+$, $\mu(\cdot)\ {\rm mod}\ 4q$ factorizes to a well-defined mapping $\mu_q: Sp_{q}(2n)\rightarrow \mathbb{Z}_{4q}$, that is for $S_q \in Sp_q(2n)$ the expression 
\[
\mu_{L,q}(S_{q})=\mu(\tilde S \tilde L, \tilde L)\ mod\  4q
\]
so that $\pi^{Sp}_q(\tilde S)=S_{q}$ does not depend on the choice of $\tilde S \in \widetilde{Sp}(2n)$.
\end{lemma}
\begin{proof}
The proof is given in Gosson's book \cite{gosson2} and follows directly by invoking the properties (\ref{invariance}) and (\ref{generator}) of $\mu$ on $\widetilde{Lag}(n)^2$ and by noting that these together with (\ref{doubling}) imply for $r \in \mathbb{Z}$ and $\tilde S \in \widetilde{Sp}(2n)$ and with $\alpha=(I,2\pi)$ generating $\pi_1(Sp(2n))$ as above
\[
\mu_L(\alpha^r\tilde S)=\mu_L(\alpha^r\tilde S)+4r.
\]
\end{proof}
Combining the preceding Lema and Lemma \ref{cappelmaslov}, we arrive at 
\begin{folg}
Let $S:[0,1] \rightarrow Sp(2n)$ be piecewise smooth, $S(0)=Id$, let $L\in Lag(n)$ be arbitrary and let $\hat S:[0,1]\rightarrow Mp(2n)=Sp_{2}(2n)$ be the unique lift of $S$ that begins at $Id\in Mp(2n)$, that is $\rho(\hat S(t)):=\rho_{2}(\hat S(t))=S(t), \ t\in [0,1]$ and $\hat S(0)=Id$. Denote $L(t)=S(t)L \in Lag(n), \ t\in [0,1]$. Then one has
\begin{equation}\label{masloveq}
\mu_{L,2}(\hat S(1))=2\mu_{CLM}([L(1)],L(t)) +(n-{\rm dim}(L(0)\cap L(1))\ {\rm mod}\ 8,
\end{equation}
where $\mu_{CLM}$ is the Maslov index on pairs of Lagrangian paths introduced in (\ref{CLMdef}) in $(\mathbb{R}^{2n},\omega_0)$.
\end{folg}
\begin{proof}
Let $\tilde L \in \widetilde{Lag}(n)$ be any element covering $L(0)$, so $\pi(\tilde L)=L(0)$. Let $\tilde S$ be the element in $\widetilde{Sp}(2n)$ defined by the homotopy class of $S:[0,1] \rightarrow Sp(2n),\  S(0)=Id$, then by (\ref{commutation}) we have
\begin{equation}\label{proj1}
\pi^{Sp}_2(\tilde S)=\hat S(1).
\end{equation}
On the other hand, denote the lift of $S(t)$ to $\widetilde{Sp}(2n)$ by $\tilde S:[0,1]\rightarrow \widetilde{Sp}(2n)$. Then we have since $\tilde S(1)=\tilde S$ that $\tilde S(t) \tilde L \in \widetilde{Lag}(n)$ connects $\tilde L$ in $\widetilde{Lag}(n)$ to $\tilde S\tilde L$ and projects to $S(t)L$, that is $\pi(\tilde S(t)\tilde L)=S(t)L\in Sp(2n)$. Using (\ref{proj1}) and the latter observations together with (\ref{maslovequality}), Lemma \ref{cappelmaslov}, Lemma \ref{modmaslov} and the relation between $\mu_{CLM}$ and $\mathcal{M}(\cdot,\cdot)$ expressed in (\ref{CLMdef}) we arrive at (\ref{masloveq}).
\end{proof}
\section{The Metaplectic Representation and the Symplectic Spinor bundle}\label{metaplect}
As we saw in the last section, $\pi_1({Sp}(2n))=\mathbb{Z}$, this implies since there is only one conjugation class of subgroups of index $2$ in $\mathbb{Z}$, that there is up to isomorphism exactly one connected two-fold covering $\rho:Mp(2n) \rightarrow Sp(2n)$, fitting into the sequence
\[
1 \quad \rightarrow \quad \mathbb{Z}_{2} \quad \rightarrow \quad
Mp(2n) \quad \xrightarrow{\rho}  \quad Sp(2n) \quad
\rightarrow \quad 1.
\]
Following Weil, Segal and Shale (\cite{vergne}, \cite{wallach1}), $Mp(2n)\simeq Sp_2(2n)$ (we will prefer the notation $Mp(2n)$ in the context of its 'metaplectic' representation, described in what follows) admits a unitary, faithful representation $\kappa:Mp(2n)\rightarrow \mathcal{U}(L^2(\mathbb{R}^n))$. This can be constructed by lifting the projective representation of $Sp(2n)$ induced by intertwining the Schroedinger representation of the Heisenberg group to $Mp(2n)$. The representation $\kappa$ has the following explicit construction on the elements of three generating subgroups of $Mp(2n,\mathbb{R})$, as follows:
\begin{enumerate}
\item Let $g(A)=(det(A)^{\frac{1}{2}},\left(\begin{smallmatrix}A&0
\\ 0&(A^{t})^{-1}\end{smallmatrix} \right))$ where $A \in GL(n,\mathbb{R})$. To fix a root of $det(A)$ defines $g(A)$ as an element in $Mp(2n)$ and we have
\begin{equation}\label{metalinear}
(\kappa(g(A))f)(x)= det(A)^{\frac{1}{2}}f(A^{t}x), \ f \in L^2(\mathbb{R}^n).
\end{equation}
\item Let $B \in M(n,\mathbb{R})$ s.t. $B^{t} = B$, set $t(B)
= \left(\begin{smallmatrix}1&B \\0&1\end{smallmatrix} \right) \in
Sp(2n)$, then the set of these matrices is simply-connected. So $t(B)$ can be considered an element of
$Mp(2n)$, with $t(0)$ being the identity in $Mp(2n)$. Then one has
\begin{equation}\label{meta2}
(\kappa(t(B))f)(x) = e^{-\frac{i}{2}\langle Bx,x\rangle}f(x).
\end{equation}
\item Fixing the root $i^{\frac{1}{2}}$, we can consider
$\sigma=(i^{\frac{1}{2}},\left(\begin{smallmatrix}0&-1
\\1&0\end{smallmatrix}\right))$ as an element of $Mp(2n)$. Then
\begin{equation}\label{fourier}
(\kappa(\sigma)f)(x)=(\frac{i}{2\pi})^{\frac{n}{2}}\int_{\mathbb{R}^{n}}e^{i\langle
x,y\rangle}f(y)dy,
\end{equation}
so $\kappa(\sigma)= i^{\frac{n}{2}}F^{-1}$, where $F$ is the usual Fourier transform.
\end{enumerate}
Inspecting these formulas it is obvious that the metaplectic group $Mp(2n)$ acts bijectively on the Schwartz space $\mathcal{S}(\mathbb{R}^{n})$, so its closure extends to $\mathcal{U}(L^2(\mathbb{R}^n))$. We then have the following Theorem due to Wallach \cite{wallach1} (p. 193, Theorem 4.53) which gives a description of $\kappa$ on a certain subset of $Mp(2n)$ in terms of oscillatory integrals:
\begin{theorem}
Let $\mathcal{A}=\left(\begin{smallmatrix}A&B\\C&D\end{smallmatrix}\right)\in Sp(2n)$ s.t. ${\rm det}\ B\neq 0$, then
\begin{equation}\label{quadraticfourier}
\hat S_{W,m}f(x):=\kappa(({\rm det}^{-1/2}B,\mathcal{A}))f(x)=(\frac{1}{2\pi i})^{n/2}|{\rm det}(B)|^{-\frac{1}{2}}i^m\int_{\mathbb{R}^n}e^{2\pi iW(x,x')}f(x')dx',
\end{equation}
where $W(x,x')$ is the generating function associated to the quadratic form given by $\mathcal{A}$, that is $(x,p)=\mathcal{A}(x',p')$ if and only if $p=\partial_x W(x,x')$, $p'=-\partial_{x'}W(x,x')$.
\end{theorem}
Note that here we fixed the root $i^{n/2}=(e^{i\pi/4})^n$ while the choice of the root ${\rm det}^{-1/2}(B)$ fixes the element in $Mp(2n)$ covering $\mathcal{A}$. Denote now $\hat J:=\kappa(\sigma^{-1})= i^{-\frac{n}{2}}F$, where we fix again $i^{n/2}=(e^{i\pi/4})^n$. Furthermore, write for $A$ as in (\ref{metalinear}) $\kappa(g(A),m)= |{\rm det}(A)|^{1/2}i^{m}f(A^{t}x)$, where $m\in \mathbb{Z}$ and $|{\rm det}(A)|^{1/2}$ (as already in (\ref{quadraticfourier})) denotes the positive root of $|{\rm det}(A)|$. Then for $P,Q \in {\rm M}(n,\mathbb{R})$, s.t. $P=P^t$, $Q=Q^t$ and $L\in {\rm GL}(n,\mathbb{R})$ we define the quadratic form
\begin{equation}\label{generatingfunc}
W(x,x')=\frac{1}{2}\langle Px,x\rangle-\langle Lx,x'\rangle +\frac{1}{2}\langle Qx',x'\rangle ,
\end{equation}
where $\langle\cdot,\cdot\rangle$ denotes the standard scalar product on $\mathbb{R}^n$. We will use the notation $W=(P,L,Q)$ to refer to a quadratic form of the form (\ref{generatingfunc}) in the following. Then, by a result of de Gosson (\cite{gosson}, Prop. 7.2), the 'quadratic Fourier transform' $\hat S_{W,m}$ can be decomposed as
\begin{equation}\label{qdecomp}
\hat S_{W,m}=\kappa(({\rm det}^{1/2}(L),S_W)=\kappa(t(P))\kappa(g(L),m)\hat J\kappa(t(Q)),\ {\rm where}\ S_{W}:=\left(\begin{smallmatrix}L^{-1}Q&L^{-1}\\PL^{-1}Q-L^T&L^{-1}P\end{smallmatrix}\right),
\end{equation}
where here, $|{\rm det}(L)|^{1/2}i^m={\rm det}^{1/2}(L)$. The next Theorem identifies the Maslov index $\mu_{L_0,2}$ on $Sp_2(2n)$, where $L_0=\{0\}\times \mathbb{R}^n$, as introduced in Lemma \ref{modmaslov}, with an index defined on the group generated by the set $\hat S_{W,m}$, for $W$ as in (\ref{generatingfunc}). This group turns out to be $Mp(2n)$.
\begin{theorem}\label{sympdecomp}
The image $\kappa(Mp(2n))\subset \mathcal{U}(L^2(\mathbb{R}^n))$ is generated by the set $\hat S_{W,m}$, $W$ being of the form (\ref{generatingfunc}). Any element $\hat S\in Mp(2n)$ can be (non-uniquely) written as 
\begin{equation}\label{decomp2}
\hat S=\hat S_{W,m}\hat S_{W',m'},
\end{equation}
where $W,W'$ are of the form (\ref{generatingfunc}). Then setting $\hat \mu(\hat S_{W,m})=2m-n\ {\rm mod}\ 8$ for any 'quadratic Fourier transform' $\hat S_{W,m}$ as defined in (\ref{quadraticfourier}), the integer
\begin{equation}\label{decomp3}
\hat \mu(\hat S):=\hat \mu(\hat S_{W,m})+\hat \mu(\hat S_{W',m'}) +\widehat{{\rm sign}}(P'+Q)
\end{equation}
where $(\hat \cdot)$ denotes the image in $\mathbb{Z}_8$ and ${\rm sign}$ the signature of a quadratic form, is well-defined and independent of the choice of $(W,m), (W',m')$. Furthermore, assuming that $\hat S\in \kappa(Mp(2n))$ maps to $S_2 \in Sp_2(2n)$ w.r.t. the identification $\kappa(Mp(2n))\simeq Sp_2(2n)$, we have
\begin{equation}\label{masloveq2}
\hat \mu(\hat S)=\mu_{L_0,2}(S_2),
\end{equation}
using the index $\mu_{L_0,2}:Sp_2(2n)\rightarrow \mathbb{Z}_8$ introduced in Lemma \ref{modmaslov}.
\end{theorem}
\begin{proof}
That $\kappa(Mp(2n))$ is generated by the 'quadratic Fourier transforms' $\hat S_{W,m}$ follows immediately from the decomposition (\ref{qdecomp}) and the formulas given for $\kappa$ in (\ref{metalinear}) to (\ref{fourier}). All other assertions, namely (\ref{decomp2}), (\ref{decomp3}) and (\ref{masloveq2}) are proven by Gosson in \cite{gosson} (Prop. 7.2, Theorem 7.22 and Corollary 7.30, respectively).
\end{proof}
Let now $(M,\omega)$ be a symplectic manifold of dimension $2n$. For $p\in M$ we denote by $R_p$ the set of symplectic bases in $T_pM$, that is the $2n$-tuples $e_{1}, \dots,e_{n},f_{1}, \dots, f_{n}$
so that
\[
\omega_{x}(e_{j},e_{k})=\omega_{x}(f_{j},f_{k}) = 0, \ 
\omega_{x}(e_{j},f_{k}) = \delta_{jk} \quad \textrm{for} \ j,k =1,\dots,2n.
\]
The symplectic group $Sp(2n)$ acts simply transitively on $R_p, \ p \in M$ and we denote by $\pi_{R}:R:=\bigcup_{p\in m} R_p \rightarrow M$ the symplectic frame bundle. By the Darboux Theorem $R$ it is a  locally trivial $Sp(2n)$-principal fibre bundle on $M$. As it is well-known (see \cite{duff}), the $\omega$-compatible almost complex structures $J$ are in bijective correspondence with the set of $U(n)$-reductions of $R$. Given such a $J$, we call local sections of the associated $U(n)$-reduction $R^J$ of the form $(e_1,\dots, e_n, f_1,\dots,f_n)$ unitary frames. These frames are characterized by
\[
g(e_j,e_k)=\delta_{jk} \quad g(e_j,f_k)=0, \quad Je_j=f_j,
\]
where $j,k=1,\dots,n$ and $g(\cdot,\cdot)=\omega(\cdot,J\cdot)$. Now a metaplectic structure of $(M,\omega)$ is a $\rho$-equivariant $Mp(2n)$-reduction of $R$, that is:
\begin{Def}\label{metaplecticst}
A pair $(P,f)$, where $\pi_P:P\rightarrow M$ is a $Mp(2n,\mathbb{R})$-principal bundle on $M$ and $f$ a bundle morphism $f:P \rightarrow R$, is called metaplectic structure of $(M,\omega)$, if the following diagram commutes:
\begin{equation}\label{metadiag}
\begin{CD}
P \times Mp(2n,\mathbb{R})  @>>> P \\
@VV{f\times\rho}V              @VV{f}V \\
R \times Sp(2n,\mathbb{R}) @>>>   R
\end{CD}
\end{equation}
where the horizontal arrows denote the respective group actions.
\end{Def}
It follows that $f:P\rightarrow R$ is a two-fold connected covering. Furthermore it is known (\cite{habermann}, \cite{kost}) that $(M,\omega)$ admits metaplectic structure if and only if $c_{1}(M)=0 \ mod \ 2$. In that case, the isomorphism classes of metaplectic structures are classified by $H^{1}(M,\mathbb{Z}_{2})$. $\kappa$ defines a continuous left-action of $Mp(2n,\mathbb{R})$ on $L^{2}(\mathbb{R}^n)$, acting unitarily on $L^{2}(\mathbb{R}^{n})$. Combining this with the right-action of $Mp(2n)$ on a fixed metaplectic structure $P$, we get a continuous right-action on $P\times L^{2}(\mathbb{R}^n)$ by setting
\[
\begin{split}
(P \times L^{2}(\mathbb{R}^{n})) \times Mp(2n) \ &\rightarrow \
P \times L^{2}(\mathbb{R}^{n}) \\
((p,f),g) \ &\mapsto \ (pg,\kappa(g^{-1})f).
\end{split}
\]
The symplectic spinor bundle $\mathcal{Q}$ is defined to be its orbit space
\[
\mathcal{Q}=P \times_{\kappa}L^{2}(\mathbb{R}^{n}) := (P \times
L^{2}(\mathbb{R}^{n}))/Mp(2n)
\]
w.r.t. this group action, so $\mathcal{Q}$ is the $\kappa$-associated vector bundle of $P$. We will refer to its elements in the following by  $[p,u]$, $p\in P$, $u\in L^{2}(\mathbb{R}^{n})$. Note that if $\pi_P$ is the projection $\pi_P: P \rightarrow M$ in $P$, then $\mathcal{Q}$ is a locally trivial fibration $\tilde \pi: \mathcal{Q} \rightarrow M$ with fibre $L^{2}(\mathbb{R}^{n})$ by setting $\tilde \pi([p,u])= x$ if $\pi_P(p) = x$. Note that continuous sections $\phi$ in $\mathcal{Q}$ correspond to continuous $Mp(2n)$-equivariant mappings $\hat \phi:P\rightarrow L^2(\mathbb{R}^n)$, that is $\hat \phi(pq)=\kappa(q^{-1})\hat \phi(p)$ for $p\in P$. Hence we define {\it smooth} sections $\Gamma(\mathcal{Q})$ in $\mathcal{Q}$ as the continuous sections whose corresponding mapping $\hat \phi$ is smooth as a map $\hat \phi:P\rightarrow L^2(\mathbb{R}^n)$. It then follows (\cite{habermann}) that $\hat \phi(p)\in \mathcal{S}(\mathbb{R}^n)$ for all $p\in P$, so smooth sections in $\mathcal{Q}$ are in fact sections of the subbundle
\[
\mathcal{S} = P \times_{\kappa}\mathcal{S}(\mathbb{R}^{n}).
\]
Given a $U(n)$-reduction $R^J$ of $R$ w.r.t. a compatible almost complex structure $J$ on $M$ and a fixed metaplectic structure $P$, we get a $\hat U(n):=\rho^{-1}(U(n))$-reduction $\pi_{P^J}:P^J \rightarrow M$ of $P$ by setting $P^J:=f^{-1}(R^J)$, where $f$ is as in Definition \ref{metaplecticst}. So we get by denoting the restriction of $\kappa$ to $\hat U(n)$ by $\tilde \kappa$ an isomorphism of vectorbundles
\begin{equation}\label{spinorbundle}
\mathcal{Q}\simeq \mathcal{Q}^J:=P^J \times_{\tilde \kappa}L^{2}(\mathbb{R}^{n}).
\end{equation}
Correspondingly we define $\mathcal{S}^J$ so that $\mathcal{S}^J\simeq \mathcal{S}$. At this point, the Hamilton operator $H_0$ of the harmonic oscillator on $L^2(\mathbb{R}^n)$ gives rise to an endomorphism of $\mathcal{S}$ and a splitting of $\mathcal{Q}$ into finite-rank subbundles as follows. Let $H_0:\mathcal{S}(\mathbb{R}^n)\rightarrow \mathcal{S}(\mathbb{R}^n)$ be the Hamilton operator of the $n$-dimensional harmonic oscillator as given by
\[
(H_0u)(x)=-\frac{1}{2}\sum_{j=1}^{n}(x_{j}^{2}u-\frac{\partial^{2}u}{\partial x_{j}^{2}}), \ u \in \mathcal{S}(\mathbb{R}^n).
\]
\begin{prop}[\cite{habermann}]\label{decomp}
The bundle endomorphism $\mathcal{H}^J:\mathcal{S}^J\rightarrow\mathcal{S}^J$ declared by $\mathcal{H}^J([p,u])=[p,H_0u],\ p\in P, u\in \mathcal{S}(\mathbb{R}^n)$ is well-defined. Let $\mathcal{M}_l$ denote the eigenspace of $H_0$ with eigenvalue $-(l+\frac{n}{2})$. Then the spaces $\mathcal{M}_l,\ l \in \mathbb{N}_0$ form an orthogonal decomposition of $L^2(\mathbb{R}^n)$ which is $\tilde \kappa$-invariant. So $\mathcal{Q}^J$ decomposes into the direct sum of finite rank-subbundles
\[
\mathcal{Q}^J_l=P^J \times_{\tilde \kappa}\mathcal{M}_l, \quad {\rm s.t.}\ {\rm rank}_{\mathbb{C}}\mathcal{Q}^J_{k}=\left(\begin{matrix}n+k-1 \\ k \end{matrix}\right)
\]
where we defined $\mathcal{Q}^J_l=\{q \in \mathcal{S}:\mathcal{H}^J(q)=-(l+\frac{n}{2})q\}$. 
\end{prop}
\begin{proof}
It is well-known (see \cite{wallach1}, \cite{habermann}) that $H_0$ can be identified with the unique element $j \in {\mathfrak mp}(2n)$ that satisfies $\rho_*(j)=-J \in {\mathfrak sp}(2n)$. Here, $J$ denotes the standard complex structure on $\mathbb{R}^{2n}$ and  ${\mathfrak mp}(2n)$ the Lie algebra of $Mp(2n)$. Then one sees that $J$ commutes with all elements of the Lie algebra of $U(n)$, as given by
\begin{equation}\label{lie}\mathfrak{u}(n)=\lbrace X \in
\mathfrak{gl}(2n,\mathbb{R}):XJ=JX,\ X^{t}+X=0\rbrace. 
\end{equation}
Consequently, $H_0$ factors to a bundle endomorphism $\mathcal{H}^J$ and the other assertions follow from known results on the eigenspaces of $H_0$ on $L^2(\mathbb{R}^n)$ (see \cite{wallach1}).
\end{proof}
To prove Theorem \ref{theorem2}, we will have to define the dual spinor bundle $Q'$ of $Q$. To do this, note that if we topologize the Schwartz space $\mathcal{S}(\mathbb{R}^{n})$ by the countable family of semi-norms
\[
p_{\alpha,m}(f)={\rm sup}_{x\in \mathbb{R}^n}(1+|x|^m)|(D^\alpha f)(x)|, \  f\in \mathcal{S}(\mathbb{R}^{n}),
\]
then the topology of $(\mathcal{S}(\mathbb{R}^{n}),\tau)$ is induced by a translation-invariant complete metric $\tau$, hence manifests the structure of a Frechet-space. Furthermore $\kappa:Mp(2n)\rightarrow \mathcal{U}(\mathcal{S}(\mathbb{R}^{n}))$ still acts continuously, which follows by the decomposition (\ref{metalinear})-(\ref{fourier}) and the fact that multiplication by monomials and Fourier transform act continuously w.r.t. $\tau$, which is a standard result (see \cite{mueller}). Then, denoting the dual space of $(\mathcal{S}(\mathbb{R}^{n}),\tau)$ as $\mathcal{S}'(\mathbb{R}^{n})$, we can consider for any pair $T \in \mathcal{S}'(\mathbb{R}^{n}), g \in Mp(2n)$ the continuous linear functional $\hat \kappa(g)(T) \in \mathcal{S}'(\mathbb{R}^{n})$ defined by
\begin{equation}\label{metadual}
(\hat \kappa(g)(T))(f)=T(\kappa(g)^*f), \ f \in \mathcal{S}(\mathbb{R}^{n}).
\end{equation}
Thus we have an action $\hat \kappa:Mp(2n)\times \mathcal{S}'(\mathbb{R}^{n})\rightarrow \mathcal{S}'(\mathbb{R}^{n})$ which extends $\kappa:Mp(2n)\rightarrow \mathcal{U}(\mathcal{S}(\mathbb{R}^{n}))$ and is continuous relative to the weak-$*$-topology on $\mathcal{S}'(\mathbb{R}^{n})$. Note that since the inclusion $i_1:\mathcal{S}(\mathbb{R}^{n})\subset L^2(\mathbb{R}^n)$ is continuous, we have the continuous triple of embeddings $\mathcal{S}(\mathbb{R}^{n})\subset L^2(\mathbb{R}^n) \subset \mathcal{S}'(\mathbb{R}^{n})$. Here $L^2(\mathbb{R}^n)$ carries the norm topology and the inclusion $i_2:L^2(\mathbb{R}^n) \hookrightarrow \mathcal{S}'(\mathbb{R}^{n})$ is given by $i_2(f)(u)=(f,\overline u)_{L^2(\mathbb{R}^n)}$ where the latter denotes the usual $L^2$-inner product on $\mathbb{R}^n$. We thus define in analogy to (\ref{spinorbundle})
\[
\mathcal{Q}'=P^J \times_{\hat \kappa}\mathcal{S}'(\mathbb{R}^{n}),
\]
where here, $\hat \kappa:U(n)\rightarrow {\rm Aut}(\mathcal{S}'(\mathbb{R}^{n}))$ means restriction of $\hat \kappa$ to $U(n)$ (using the same symbol). Now any fixed section $\varphi \in \Gamma(\mathcal{Q}')$ may be evaluated on any $\psi \in \Gamma(\mathcal{Q})$ by writing $\varphi=[\overline s, T], \psi=[\overline s, u]$ w.r.t. a local section $\overline s: U\subset M\rightarrow P^J$ and smooth mappings $T:U\rightarrow \mathcal{S}'(\mathbb{R}^{n})$, $u: U\rightarrow \mathcal{S}(\mathbb{R}^{n})$ by setting 
\[
\varphi(\psi)|U(x)= T(u)(x), \ x \in U\subset M.
\]
It is clear that this extends to a mapping $\varphi:\Gamma(\mathcal{Q})\rightarrow C^\infty(M)$. Furthermore, for any $p \in P^J, x\in M$ s.t. $\pi_{P^J}(p)=x$, we can define $\delta_{p} \in \mathcal{Q}'_x$ which assigns to any $\psi =[p, u]\in \mathcal{Q}_x$ the value
\begin{equation}\label{delta}
\delta_{p}(\psi):= \delta(0)(u),
\end{equation}
where $\delta(0)\in \mathcal{S}'(\mathbb{R}^{n})$ is the linear functional $\delta_0(u)=u(0), \ u \in \mathcal{S}(\mathbb{R}^n)$. Note that $\delta_{p}$ depends on $p$ and, unless $P^J$ has a global section, there is not necessarily a smooth extension of $\delta_{p}$ to an element of $\Gamma(\mathcal{Q}')$ so that over any point $x \in M$ (\ref{delta}) holds for some $p \in P^J_x$. Nevertheless, given some connection $\overline Z:P^J\rightarrow \mathfrak{u}(n)$ the associated parallel transport $\mathcal{P}^{\overline Z}_\gamma(t):P^J_{\gamma(0)}\rightarrow P^J_{\gamma(t)},\ t \in [0,1]$ along $\gamma:[0,1]\rightarrow M$ enables to extend $\delta_{p}$ along $\gamma$ to a section $\delta_{\gamma,p}\in \Gamma(\gamma^*(\mathcal{Q}'))$ by setting
\[
\delta_{\gamma,p}(t)=[\mathcal{P}^{\overline Z}_{\gamma}(t)(p),\delta(0)] \in \gamma^*(\mathcal{Q}')_{\gamma(t)}.
\]
Here, $\gamma^*(\mathcal{Q}')$ denotes the pull-back by $\gamma:[0,1]\rightarrow M$ of the bundle $\mathcal{Q}'$.

\section{Proof of the Theorems}\label{proofs}
In the following, let $(M,\omega)=(\mathbb{R}^{2n},\omega_0)$, using the notation from Section \ref{maslov} and let $i:L \hookrightarrow \mathbb{R}^{2n}$, be an embedded Lagrangian submanifold, that is $i^*\omega=0$. Denote $J$ the standard complex structure on $\mathbb{R}^{2n}$ and $Q^J$ the symplectic spinor bundle associated to the $\hat U(n)$-reduction $f^J:P^J\rightarrow R^J$ of the trivial metaplectic structure $f:P \rightarrow R$ on $\mathbb{R}^{2n}$ (note that since $c_1(\mathbb{R}^{2n},\omega)=0$, there is only this structure up to isomorphism). We first note that the bundles $\pi^L_R:i^*R^J\rightarrow L$ resp. $\pi^L_P:i^*P^J\rightarrow L$ allow a further reduction to $O(n)$ resp. $\hat O(n)=\rho^{-1}(O(n))$ induced by the inclusion
\begin{equation}\label{orth}
\hat i:O(n)\hookrightarrow U(n)=Sp(2n)\cap O(2n),\quad A \mapsto \left(\begin{smallmatrix}A
&0\\0&A\end{smallmatrix} \right),
\end{equation}
where $A \in M(n,\mathbb{R}),\ A^tA=I$. Denote $g(\cdot,\cdot)=\omega(\cdot,J\cdot)$ the metric induced on $L$ by $(J,\omega)$, which is simply the restriction of the standard metric to $L$.
\begin{lemma}\label{onred}
There is an $O(n)$-reduction $(\hat R_L,\pi^L_{R},L,O(n))$ of the principal bundle $(i^*R^J, \pi_R|L, L, U(n))$ which is induced by the inclusion (\ref{orth}). This reduction gives rise to an $\hat O(n)=\rho^{-1}(O(n))$-reduction $(\hat P_L,\pi_{P}^L, L, \hat O(n))$ of $(i^*P^J, \pi_P|L, L, \hat U(n))$ on $L$ by setting $\hat P_L=f^{-1}(\hat R_L)$ so that the diagram
\begin{equation}\label{metadiag2}
\begin{CD}
\hat P_L \times \hat O(n)  @>>> \hat P_L \\
@VV{\hat f\times \hat \rho}V              @VV{\hat f}V \\
\hat R_L \times O(n) @>>>   \hat R_L
\end{CD}
\end{equation}
commutes. Here, $\hat f$ and $\rho$ denote appropriate restrictions of $f:P\rightarrow R$ and $\rho:Mp(2n)\rightarrow Sp(n)$ as defined above.
\end{lemma}
\begin{proof}
To show that the $U(n)$-bundle $\hat R:=i^*R^J$ over $L$ allows the asserted $O(n)$-reduction, we have to show that the bundle $\hat R\times_{U(n)}U(n)/O(n)$ allows a global section over $L$. But this is determined by setting locally for $v \in U\subset L$
\[
\phi(x)=[s(x), 1],\ x \in L,
\]
where $s(x)\in \hat R, \pi_R(s(x))=x$ and $s(x)=(e_1(x),\dots,e_n(x), Je_1,\dots,e_n(x))$, where $(e_1,\dots, e_n)$ is some local orthonormal basis on $L$ and $1$ denotes the identity in $U(n)/O(n)$. It is clear that $\phi$ declares a well-defined globally non-vanishing section of $\hat R\times_{U(n)}U(n)/O(n)$.
\end{proof}
Let now $Z^L:TR_L \rightarrow \mathfrak{o}(n,\mathbb{R})$ be the connection on $R_L$, where $R_L$ is the $O(n)$-bundle of orthonormal frames on $(L,g)$, which corresponds to the Levi-Civita covariant derivative $\nabla^g$ on $(L,g)$. Then it is clear that if $j:R_L\rightarrow \hat R_L$ is the fibre bundle isomorphism given by setting for any $x\in L$ $j(e_1,\dots, e_n)=(e_1,\dots,e_n,Je_1,\dots,Je_n)$, where $(e_1,\dots,e_n)\in (R_L)_x, \ x\in L$ is an orthonormal basis in $T_xL$, that
\begin{equation}\label{leviZ}
Z:T\hat R_L\rightarrow \mathfrak{o}(n), \quad Z:=\hat i_*\circ Z^L\circ(j_*)^{-1},
\end{equation}
defines a well-defined connection on $\hat R_L$. Furthermore, $Z$ lifts to a connection $1$-form $\overline Z: T\hat P_L \rightarrow \hat {\mathfrak{o}}(n)$, so that the following diagram commutes, here we set $\hat {\mathfrak{o}}(n)=\rho_*^{-1}(\mathfrak{o}(n))$:
\[
\begin{CD}
T\hat P_L   @>{\overline Z}>> \hat {\mathfrak{o}}(n) \\
@VV{\hat f_{*}}V              @VV{\hat \rho_{*}}V \\
T\hat R_L  @>Z>>   		\mathfrak{o}(n)
\end{CD}
\]
Since $\rho_{*}$ is an isomorphism, we can actually define $\overline Z$ as $\overline Z= \rho_{*}^{-1}\circ
Z\circ f_{*}$ on $T\hat P_L$. Note that using the above, $i^*\mathcal{Q}^J_l, \ l\in \mathbb{N}_0$ can be written as
\[
i^*\mathcal{Q}^J_l=\hat P_L \times_{\hat \kappa}\mathcal{M}_l,
\]
where $\hat \kappa=\kappa|\hat O(n)$. For $s:U\subset L\rightarrow \hat R_L$ a local section in $\hat R_L$, let $\overline s:U\subset L\rightarrow P_L$ be a lift to $\hat P_L$. Then if $X\in \Gamma(TL)$, $\overline Z$ induces a covariant derivative in $\Gamma(i^*\mathcal{Q}^J)$ by setting for a local section $\varphi=[\overline s,u]$, where $u:U\rightarrow L^{2}(\mathbb{R}^{n})$
\[
\nabla_X\varphi=[\overline s,du(X)+\hat \kappa_*(\overline Z\circ\overline s_*(X))u].
\]
On the other hand, given a path $\gamma:[0,1]\rightarrow L$, $\gamma(0)\in U$, the horizontal lift $\gamma_p$ of $\gamma$ w.r.t. $\overline Z$ and a given $\gamma_p(0)=p\in (P_L)_{\gamma(0)}$ defines a map $\mathcal{P}^{\overline Z}_\gamma(t):(\hat P_L)_{\gamma(0)}\rightarrow (\hat P_L)_{\gamma(t)}$ by setting $\mathcal{P}^{\overline Z}_\gamma(t)(p)=\gamma_p(t)$. This defines the notion of parallel transport $\mathcal{P}^\nabla_\gamma(t): (i^*\mathcal{Q}_0^J)_{\gamma(0)}\rightarrow (i^*\mathcal{Q}^J_0)_{\gamma(t)}$ by setting 
\[
\mathcal{P}^\nabla_\gamma(t)[p,u]=[\mathcal{P}^{\overline Z}_\gamma(t)(p),u], \ u\in \mathcal{M}_0,\  t\in [0,1].
\]
Further, if $\varphi\in \Gamma(i^*\mathcal{Q}^J_0)$ and $\gamma'(0)=X$, then $\nabla_X\varphi=\frac{d}{dt}(\mathcal{P}^\nabla_\gamma(t)(\varphi(\gamma(t))))|_{t=0}$. Note that here, $u \in \mathcal{M}_0$ can be chosen to span the one-dimensional subspace $\mathcal{M}_0\subset L^2(\mathbb{R}^n)$ (see Lemma \ref{decomp}), hence $u(x)=e^{-\frac{(x,x)}{2}}, \ x\in \mathbb{R}^n$.\\ 
Now to prove Theorem \ref{theorem1}, let $x \in U\subset L$ and $s,\overline s$ as chosen above. Let $\gamma:[0,1]\rightarrow L$ be any closed smooth path in $L$ with basepoint $x$, that is $\gamma(0)=\gamma(1)=x$ and let $\mathcal{P}^{Z}_\gamma(t):(\hat R_L)_{\gamma(0)}\rightarrow (\hat R_L)_{\gamma(t)}$ be the parallel transport in $\hat R_L$ induced by $Z$. It follows that there is a unique smooth path
\[
S:[0,1]\rightarrow U(n)=Sp(2n)\cap O(2n) \ {\rm s.t.}\ \mathcal{P}^{Z}_\gamma(t)(s(x))=S(t).s(x),
\]
so that $S(0)=Id$, where we used the trivialization of $i^*R^J$ induced by the Euclidean connection $\nabla^0$ on $T\mathbb{R}^{2n}$ to compare $\mathcal{P}^{Z}_\gamma(t)(s(x))$ and $s(x)$ for any $t \in [0,1]$ in $i^*R^J$. Analogously we have a path
\[
\hat S:[0,1]\rightarrow \hat U(n) \ {\rm s.t.}\ \mathcal{P}^{\overline Z}_\gamma(t)(\overline s(x))=\hat S(t).\overline s(x),
\]
so that $\rho_2(\hat S(t))=S(t),\ t \in [0,1]$ and $\hat S(0)=Id_{Mp(2n)}$. Here again, we used the triviality of $i^*P^J$ induced by the Euclidean connection $\nabla^0$ on $T\mathbb{R}^{2n}$ to compare $\mathcal{P}^{\overline Z}_\gamma(t)(s(x))$ and $s(x)$ in $i^*P^J$. By the construction of $Z$ in (\ref{leviZ}), it follows that $S(1) \in \hat i(O(n))$, where $\hat i:O(n)\hookrightarrow U(n)$ is the inclusion defined in (\ref{orth}). So writing $S(1)=\left(\begin{smallmatrix}A&0\\0&A\end{smallmatrix}\right)$ for $A\in O(n)$ we have using (\ref{metalinear}) and setting $s(x)=p$
\begin{equation}\label{par}
\begin{split}
\mathcal{P}^{\nabla^g}_\gamma(1)[p,u]&=[\mathcal{P}^{\overline Z}_\gamma(1)p,u]= [p,\kappa(\hat S(1))u]=[p,(\kappa(g(A),m))u(y)]\\
&=[p,{\rm det}(A)^{\frac{1}{2}}u(A^{t}y)]=[p, i^m u(y)], \ y \in \mathbb{R}^n.
\end{split}
\end{equation}
Here, as above, $m \in \mathbb{Z}_4$ is determined by requiring ${\rm det}(A)^{\frac{1}{2}}=|{\rm det}(A)|i^m$ and we used that $|{\rm det}(A)|=1$. On the other hand, since $S(1)=\left(\begin{smallmatrix}0&A\\-A&0\end{smallmatrix}\right)\left(\begin{smallmatrix}0&-1\\1&0\end{smallmatrix}\right)$ we have the decomposition
\begin{equation}\label{metalineardec}
\kappa(\hat S(1))=\hat S_{W,m'}\hat S_{W',n}, \quad {\rm where} \ W=(0,A^{t},0), \ W'=(0,-Id_{R^n},0).
\end{equation}
Here $\hat S_{W,m}, \hat S_{W',n} \in \mathcal{U}(L^2(\mathbb{R}^n))$ are as referred to in Theorem \ref{sympdecomp} and $m'\in \mathbb{Z}_4$ is determined by using (\ref{qdecomp}) and noting that $P=Q=P'=Q'=0$. Then
\begin{equation}
\begin{split}\label{u0calc}
(\hat S_{W,m'}\hat S_{W',n}u)(x)&=\kappa(t(P))\kappa(g(A^{t}),m')\hat J\kappa(t(Q))\kappa(t(P'))\kappa(g(-Id_{\mathbb{R}^n}),n)\hat J\kappa(t(Q'))u(x)\\
&=i^{n-n}\kappa(g(A^{t}),m')u(x)=i^{m'} u(x), \ x\in \mathbb{R}^n,
\end{split}
\end{equation}
where we used $\hat J u=i^{-\frac{n}{2}}u$ (recall that $u(x)=e^{-\frac{(x,x)}{2}}$) and we fixed ${\rm det}(-Id_{\mathbb{R}^n})^{\frac{1}{2}}=i^n$, so by comparing with (\ref{par}) we infer that $m=m'\in \mathbb{Z}_4$. Now, by using notation from Theorem \ref{sympdecomp}, we have $\hat \mu(\hat S(1))=\hat \mu(\hat S_{W,m'})+\hat \mu(\hat S_{W',n})= 2m-n +2n-n=2m$. Comparing that to (\ref{masloveq2}) resp. (\ref{masloveq}) we deduce that
\begin{equation}\label{bla3}
2m=\hat \mu(\hat S(1)))=\mu_{L_0,2}(\hat S(1))=2\mu_{CLM}([L(1)],L(t)) \ {\rm mod}\ 8,
\end{equation}
where $L(t)=S(t)L_0$ and $L_0=\{0\}\times \mathbb{R}^{n}$ and we used that ${\rm dim}(L(0)\cap L(1))=n$. By definition $L(t)={\rm span}\{\sum_{i=1}^{2n}S_{ij}(t)e_i\}_{j=1}^{n}$, where ${\rm span}\{e_j\}=L_0$ denotes the standard basis of $\{0\}\times \mathbb{R}^n$. Hence defining $\tilde S\in Sp(2n)$ by $\tilde Se_j=s(x)_j$ if ${\rm span}_{j=1}^n s(x)_j=T_xL$ implies that $\tilde S L(t)={\rm span}\{\sum_{i=1}^{2n}S_{ij}(t)s(x)_i\}_{j=1}^{2n}$. This means that $\tilde S L(t)=i_*(T_{\gamma(t)}L)$, if $i:L\hookrightarrow M$ is the inclusion. Using (\ref{bla3}) and the invariance of $\mu_{CLM}$ under symplectic mappings we arrive at
\[
m=\mu_{CLM}(\tilde S [L(1)],\tilde S L(t))=\mu_{CLM}(i_*(T_{x}L),i_*(T_{\gamma(t)}L)),
\]
which is by (\ref{par}) exactly the content of Theorem \ref{theorem1}.\\ 
Now to prove Corollary \ref{cor1}, note that if ${\rm Hol}_p(\overline Z)$ is the holonomy group of $p \in (\hat P_L)_x,\ x \in L$, that is
\[
{\rm Hol}_p(\overline Z)=\{g \in \hat O(n): \exists \gamma:[0,1]\rightarrow L,\ \gamma(0)=\gamma(1)=x,\ g.\gamma_p(0)=\gamma_p(1)\},
\]
then there is an identification 
\[
{\rm Par}^{\nabla^g}(i^*Q):=\{\phi \in \Gamma(i^*Q_0):\nabla\phi=0\}\simeq\{u \in \mathcal{M}_0:\kappa({\rm Hol}_p(\overline Z)u=u\}
\]
for any $p\in \hat P_L$. Since we have shown above that $\kappa({\rm Hol}_p(\overline Z))=\mathbb{Z}_4\subset S^1$ if $S^1\subset \mathbb{C}$ acts by multiplication on $\mathcal{M}_0$ and $\mathcal{M}_0$ has complex dimension one, we infer by the homotopy invariance of $\mu_{CLM}$ that ${\rm Par}^{\nabla^g}(i^*Q)=1$ if and only if $\mu_{CLM}(i_*(T_{x}L),i_*(T_{\gamma(t)}L))=0 \mod 4$ for any $\gamma\in \pi_1(L)$, which proves the Corollary.\\
To prove Theorem \ref{theorem2}, choose $p \in (\hat P_L)_x,\ x \in L$ and extend $\delta_p \in i^*(\mathcal{Q}^J)_x'=i^*\mathcal{Q}_x'$ using the connection $\nabla^0$ on $i^*\mathcal{Q}'$ induced by the canonical flat connection of $i^*(T\mathbb{R}^{2n})$ to an element $\delta_p\in \Gamma(i^*\mathcal{Q}')$. 
Now write $\delta_p(x)=[p, \delta(0)]$ and consider the parallel transport $\mathcal{P}^{\nabla^g}_\gamma$ along $\gamma:[0,1]\rightarrow L, \ \gamma(0)=x, \gamma(1)=y$ for $y\in L$ induced by the Levi-Civita-connection $\nabla^g$ of $L$ in $i^*\mathcal{Q}'$. Then
\begin{equation}\label{pararg}
\mathcal{P}^{\nabla^g}_\gamma(1)[p,\delta(0)]=[\mathcal{P}^{\overline Z}_\gamma(1)p,\delta(0)]= [p,\kappa(\hat S(1))\delta(0)]
\end{equation}
where as above, $\hat S:[0,1]\rightarrow \hat U(n)$ is determined by the requirement that $\mathcal{P}^{\overline Z}_\gamma(t)(p)=\hat S(t).p$. Here again we use the trivialization of $i^*P^J$ induced by $\nabla^0$ to consider $p$ as an element of $(i^*\mathcal{Q}')_{\gamma(t)}$ for any $t \in [0,1]$. Then $\hat S$ lifts the path $S:[0,1]\rightarrow U(n)$ determined by $\mathcal{P}^{Z}_\gamma(t)(r)=S(t).r$, where $r =\hat f(p)\in (\hat R_L)_x$. Assume now that $i_*(T_xL)\cap i_*(T_y L)=0$ in $\mathbb{R}^{2n}$, then since 
\begin{equation}\label{matrixdecomp}
S(1)=\left(\begin{smallmatrix}A
&-B\\B&A\end{smallmatrix} \right),  {\rm det}(B)\neq 0,
\end{equation}
where $A,B \in M(n,\mathbb{R}), \ A^tA+B^tB=I$ and $A^tB$ symmetric, we can write following Theorem \ref{sympdecomp} resp. Gosson \cite{gosson2} (Chapter 7.1)
\[
\kappa(\hat S(1))=\hat S_{W,m} \quad {\rm where} \ W=(P,L,Q)=(-AB^{-1},-B^{-1},-B^{-1}A), 
\]
and explicitly
\begin{equation}\label{proofdec}
\hat S_{W,m}=\kappa(t(-AB^{-1}))\kappa(g(-B^{-1}),m)\hat J\kappa(t(-B^{-1}A)).
\end{equation}
So we get by applying (\ref{metadual}) and (\ref{quadraticfourier}) and by setting $u_\epsilon=\frac{1}{\epsilon^n}u(\frac{x}{\epsilon})$ with $u$ as in (\ref{u0calc}) and $f \in \mathcal{S}(\mathbb{R}^n)$
\begin{equation}\label{eqbla}
\begin{split}
(\hat S_{W',m'}\delta(0))(f)=&{\rm lim}_{\epsilon\rightarrow 0}\int_{\mathbb{R}^n}\overline{\hat S_{W',m'}u_\epsilon} f dx\\
=&(\frac{1}{2\pi})^{n/2}i^{-m'+n/2}|{\rm det}(-B^{-1})|^{\frac{1}{2}} \int_{\mathbb{R}^n} e^{\pi i\langle Px,x\rangle}f(x)dx.
\end{split}
\end{equation}
Now note that the definition $\hat \mu(\hat S_{W,m})=2m-n$ and the formula (\ref{decomp3}) for $\hat \mu:Mp(n)\rightarrow \mathbb{Z}_4$ are compatible in the following sense. If $\hat S \in Mp(2n)$ and $\rho_2(\hat S)=S$ with $S=\left(\begin{smallmatrix}A &B\\C&D\end{smallmatrix} \right),\  {\rm det}(B)\neq 0$, then $\kappa(\hat S)=\hat S_{W,m}$ for some $W=(P,L,Q),\ m\in \mathbb{Z}_4$ and $\hat \mu(\hat S)=\hat \mu(\hat S_{W,m})=2m-n$. This follows from Theorem 7.22 (i) in Gosson's book (\cite{gosson2}). So, combining (\ref{eqbla}) with (\ref{pararg}), setting $c(y):=(\frac{1}{2\pi})^{n/2}|{\rm det}(-B^{-1})|^{\frac{1}{2}}$ and using (\ref{masloveq}) together with $2m'-n=\hat\mu(\hat S_{W',m'})$, we arrive at the following Theorem:
\begin{theorem}\label{theorem3}
Let $\gamma:[0,1]\rightarrow L$ denote a smooth path connecting $x,y \in L$ and assume that $i_*(T_xL)\cap i_*(T_yL)=0$ in $\mathbb{R}^{2n}$. Then $\mathcal{P}^{\nabla^{g}}_{\gamma}\delta_p(x)\in (i^*\mathcal{Q}')_y$ and we have
\begin{equation}
\mathcal{P}^{\nabla^{g}}_{\gamma}\delta_p(x)=c(y)e^{-i\frac{\pi}{2}\mu_{CLM}([i_{*}T_{y}L],[i_{*}\gamma])}e^{-\pi i\langle AB^{-1}\cdot,\cdot\rangle}\mathbf{1}_p(y),
\end{equation}
for $0 < c(y) \in \mathbb{R}$, $L(x)=i_*(T_xL)$ and $L(y)=i_*(T_xL)$, respectively. Here $A,B \in M(n,\mathbb{R})$ are as in (\ref{matrixdecomp}).
\end{theorem}
Now $A=0$ in the above is equivalent to $i_*(T_xL)\perp i_*(T_y L)$ which implies formula (\ref{transversal}) in Theorem \ref{theorem2}. To examine the case $i_*(T_xL)=i_*(T_y L)$, note that in this case $\kappa(\hat S(1))$ in (\ref{pararg}) decomposes as in (\ref{metalineardec}) and consequently
\begin{equation}\label{othercase4}
\begin{split}
\kappa (\hat S(1))\delta(0)&=(\hat S_{W,m'}\hat S_{W',n})\delta(0)\\
&=\kappa(g(A^{t}),m')\hat J\kappa(g(-Id_{\mathbb{R}^n}),n)\hat J \delta(0)\\
&=i^{n-n}\kappa(g(A^{t}),m')\delta(0)=i^{-m'} \delta(0). 
\end{split}
\end{equation}
Since $\hat \mu(\hat S(1))=\hat \mu(\hat S_{W,m'})+\hat \mu(\hat S_{W',n})=2m'-n+2n-n=2m'$ we can use again (\ref{masloveq}) to arrive at (\ref{othercase}).\\
To prove (\ref{transversal2}) in Corollary \ref{cor2}, note first that if ${\rm dim}\ L(x)\cap L(y)=n$, we can apply exactly the same arguments as in the proof of Theorem \ref{theorem1} above (using the trivialization of $i^*P^J$ induced by $\nabla^0$) to arrive at (\ref{transversal2}). Assume now that $L(x)\cap L(y)=0$, then if $\hat S(1) \in Mp(2n)$ is as in (\ref{pararg}) and since $\mathcal{M}_0\subset L^2(\mathbb{R}^n)$ is spanned by $u(x)=e^{-\frac{(x,x)}{2}}, \ x\in \mathbb{R}^n$, we have with $m' \in \mathbb{Z}_4$ as in (\ref{eqbla}) by what was shown in the proof of Theorem \ref{theorem3}
\begin{equation}\label{transversegen}
\begin{split}
\delta(0)(\kappa(\hat S(1))u)&=\overline{(\kappa(\hat S(1))^*\delta(0))(u)}=(\tilde c(y)i^{m'-n/2}e^{-\pi i\langle (B^t)^{-1}A^tx,x\rangle}\mathbf{1})(u)\\
&=\tilde c(y)i^{m'-n/2}\int_{\mathbb{R}^n}e^{-\pi i\langle (B^t)^{-1}A^tx,x\rangle}u(x)dx
\end{split}
\end{equation}
Here $\tilde c(y)\in \mathbb{R}^+$ and we used that $\hat S_{W,m'}^*=\hat S_{\tilde W,n-m'}$ for some quadratic form $\tilde W$ (cf. \cite{gosson2}, Prop. 7.2). For $L(x)\perp L(y)$ we have $A=0$ and the integral in (\ref{transversegen}) equals ${2\pi}^{n/2}$. Since $\hat S(1) \in \hat U(n)$, we have $\kappa(\hat S(1))u=\hat c u$ for some $\hat c \in U(1)$ so we see that $\tilde c(y){2\pi}^{n/2}=1$ and $\hat c=i^{m'-n/2}$, which gives (\ref{transversal2}) by the arguments given below (\ref{eqbla}). Finally (\ref{othercase2}) follows by noting that $\kappa(\hat U(n))(\mathcal{M}_l)\subset \mathcal{M}_l,\ l \in \mathbb{N}_0$ (see \ref{decomp}), using 
\[
\delta(0)(\kappa(\hat S(1))u)=\overline{\kappa(\hat S(1))^*\delta(0))\overline u}=i^{m'}u(0),
\]
where $u \in \mathcal{M}_l$, $\hat S(1), \ m'\in\mathbb{Z}_4$ are as in (\ref{othercase4}) and finally using (\ref{masloveq}) again.

\end{document}